\documentclass[11pt]{amsart}
\usepackage[T1]{fontenc}

\usepackage{lmodern, amsfonts,amsmath,amstext,amsbsy,amssymb,
amsopn,amsthm,upref,eucal,mathptmx}

\usepackage{mathrsfs}

\RequirePackage{xcolor} 
\definecolor{halfgray}
{gray}{0.55}
\definecolor{webgreen}
{rgb}{0,0.4,0}
\definecolor{webbrown}
{rgb}{.8,0.1,0.1}
\definecolor{red}
{rgb}{1,0,0}
\usepackage{microtype}

\usepackage[bottom]{footmisc}
\usepackage{anysize}
\usepackage{mathptmx}

\marginsize{3,5cm}{2,5cm}{2,5cm}{2,5cm}
\usepackage{amssymb}
\usepackage{graphicx,psfrag,epsfig,multirow}
\usepackage{hyperref}
\usepackage{amsmath,amsbsy,amssymb,amscd}
\usepackage{t1enc}
\usepackage[cp1250]{inputenc}
\usepackage[british]{babel}
\usepackage[all]{xy}
\usepackage{color}
\usepackage{amsfonts}
\usepackage{latexsym}
\usepackage{amsthm}
\usepackage{mathrsfs}
\usepackage{hyperref}

\newcommand \R {{ \mathbb R}}
\def\C{{\mathbb C}}
\newcommand \Z {{ \mathbb Z}}
\newcommand \N {{ \mathbb N}}
\newcommand \T {{ \mathbb T}}

\newtheorem{theorem}{Theorem}[section]
\newtheorem {lemma} [theorem]{Lemma}

\newtheorem{corollary}[theorem]{Corollary}
\newtheorem{remark}[theorem]{Remark}

\newtheorem{definition}[theorem]{Definition}

\newtheorem{problem}[theorem]{Problem}

\title[Ruelle Resonances from Cohomological Equations ]%
{ Ruelle Resonances from Cohomological Equations}

  \author{Giovanni Forni}

\address{Department  of Mathematics\\
  University of Maryland \\
  College Park, MD USA}
  
\email
    {gforni@math.umd.edu}

\date{\today}

\begin{document}

\def\echo#1{\relax}
    
\begin{abstract}
These notes are based on lectures given by the author at the Summer School on {\it Teichm\"uller dynamics, mapping class groups and applications} in Grenoble, France, in June 2018 and at the Oberwolfach Seminar on {\it Anisotropic Spaces and their Applications to Hyperbolic and Parabolic Systems} 
in June 2019. We  derive results about the so-called Ruelle resonances and the asymptotics of
correlations for several classes of systems from known results on cohomological equations and
invariant distributions for the respective unstable vector fields. In particular, we consider pseudo-Anosov diffeomorphisms on surfaces of higher genus, for horocycle flows on surfaces of constant negative curvature and for partially hyperbolic automorphisms of Heisenberg $3$-dimensional nilmanfolds. Ruelles resonances for pseudo-Anosov maps with applications to the cohomological equation for their unstable translation flows was recently studied in depth by F.~Faure, S. Gou\"ezel and E.~Lanneau~\cite{FGL} by methods based on the analysis of the transfer operator of the pseudo-Anosov map. Ruelle resonances for  geodesic flows on hyperbolic compact manifolds of any dimension and of partially hyperbolic automorphisms of Heisenberg $3$-dimensional nilmanfolds are studied by general results of  Dyatlov, Faure and Guillarmou \cite{DFG} and Faure and Tsujii \cite{FT15} based on methods of semi-classical analysis. These works do not derive results on cohomological equations for unstable flows or horospherical foliations of these systems.

 \begin{sloppypar}
  \end{sloppypar}
\end{abstract}
\maketitle

\section{Introduction}

Cohomological equations appear in several questions in dynamical systems. In fact, they arise
naturally in questions related to the triviality of {\it time changes} and the associated cocycles, and 
this connection motivates their name. They are related to linearized equations coming from conjugacy
problems, by the KAM method or other inverse function theorem. These results are classical for linear
flows on tori, and were generalized to translation flows by S.~Marmi, P.~Moussa and J.-C.~Yoccoz~\cite{MMY12}.  For renormalizable systems, obstructions to existence of solutions of the cohomological
equation are related to the phenomenon of deviation of ergodic averages for uniquely ergodic systems
 discovered by A.~Zorich \cite{Zo97} (see also M.~Kontsevich and A.~Zorich \cite{KZ97}), and later investigated 
 in the work of the author~\cite{F02} and A.~Bufetov~\cite{Bu14}.  
 A similar approach based on the theory of unitary representation extended the description of deviation of ergodic averages to other {\it renormalizable} algebraic systems, such as horocycle flows on compact hyperbolic surfaces~\cite{FF03} and Heisenberg niflows \cite{FF06} (see also  \cite{BuF14} and the surveys \cite{F14}, \cite{F15} by the author).

\smallskip
In all of the above works solutions of the cohomological equation were
constructed by methods of harmonic analysis, and thus limited to the case of systems of algebraic nature, and 
with the exception of translation flows, to homogeneous flows. This is in contrast with the world of hyperbolic 
and partially hyperbolic accessible systems where a full theory has been developed. Indeed, for the hyperbolic case the theory usually goes under the name of {\it Livsic theory}, while in the partially hyperbolic accessible 
case it was initiated  by A.~Katok and A.~Kononenko~\cite{KK96} and later developed by A.~Wilkinson \cite{W13}. 
In the uniquely ergodic, parabolic, case, the first  solution of a cohomological equation for a non-homogeneous system was given by the author in~\cite{F97} by methods of harmonic analysis. A breakthrough
came in 2005 when S.~Marmi, P.~Moussa and J.-C.~Yoccoz \cite{MMY05} were able to prove a related result on cohomological equation of interval exchange transformations by a dynamical approach based on renormalization. 
In their paper they were able to deal with the case of invariant directions of pseudo-Anosov maps, which were 
left as an open problem in~\cite{F97}.  It was the first solution of a cohomological equation by {\it dynamical methods} for a class of parabolic non-homogeneous uniquely ergodic systems.

\smallskip
Recently, the methods based on the analysis of the transfer operator developed by V.~Baladi, 
S.~Gou\"ezel, C.~Liverani and others have been refined to the point of being able to deliver dynamical analyses of the cohomological  equations of more general systems, following a conceptual scheme similar to that of
the work of Marmi, Moussa and Yoccoz~\cite{MMY05} quoted above. Roughly speaking the idea is to refine
the asymptotics on decay of correlations, or rather the analysis of the associated  transfer operator, to the 
point when it becomes possible to derive refined asymptotic estimates of the ergodic integrals. From these
asymptotic ``expansions'' for ergodic integrals the existence of solutions follows by a version of the 
Gottschalk-Hedlund theorem, according to which  functions are coboundaries if they have ``bounded'' ergodic
integrals (in the $C^0$ or $L^2$ topology). 

\smallskip
A ``proof of concept'' paper by P.~Giulietti and C.~Liverani~\cite{GL}  treats the somewhat artificial case of flows along
invariant foliations of Anosov maps of the $2$-torus. The above-mentioned work of F.~Faure, S. Gou\"ezel 
and E.~Lanneau treats the case of translation flows invariant under a pseudo-Anosov. The results of this paper 
are not surprising, and, as the author will argue in these lectures, could have been derived without too much
effort from known results on the cohomological equation, with the exception of improved loss of regularity
and the unified treatment of neutral eigenvectors. However, the methods are extremely significant, since 
they are general enough to be applied, at least in principle, to non-linear pseudo-Anosov maps and to other
general, non algebraic, systems. 

\smallskip
In these notes we are going to explain in detail, in section \ref{sec:pseudoAnosov},  how  to derive many informations about the Ruelle resonances of a pseudo-Anosov map from knowledge of the obstructions to the existence of solution to the cohomological equation for its unstable direction. 

\smallskip
A similar program, which is just outline in these notes,  in section~\ref{sec:horo}, can be 
carried out for the horocycle flow on hyperbolic (negative constant curvature) surfaces based on the
work of L.~Flaminio and the author \cite{FF03}. Ruelle resonances and asymptotic  for geodesic flows on
general hyperbolic manifolds were described in depth by S.~Dyatlov, F.~Faure and C.~Guillarmou \cite{DFG}.
For horocycle flows in variable negative curvature, partial results 
have been proved by A.~Adam \cite{Ad}  also following the Giulietti-Liverani approach.  

Finally, we outline, in section~\ref{sec:heis}, 
the computation of Ruelle resonances for partially hyperbolic automorphisms of the Heisenberg group
based on the solution of cohomological equations for Heisenberg nilflows given in \cite{FF06}. Since this
is a partially hyperbolic system, the transfer operator approach presents additional difficulties. However, 
Ruelle resonances were computed in much greater generality by F.~Faure and M.~Tsujii \cite{FT15}, who developed results of Faure \cite{Fau07}, for equivariant (isometric) extensions of symplectic Anosov diffeomorphisms  to $U(1)$ principal bundles.

\smallskip
From the point of view of parabolic flows, the translation flows stabilized by a pseudo-Anosov diffeomorphism
are exceptional among all translation flows on compact orientable surfaces, and Heisenberg nilflows stabilized
by a Heisenberg automorphism are rare among all Heisenberg nilflows. In order to develop a theory for generic (in the sense of measure) translation flows or Heisenberg nilflows, the notion of a {\it transfer cocycle} over a renormalization dynamics has to be developed and analyzed. Results in Teichm\"uller dynamics based the study of the Lyapunov structure of the so-called Kontsevich--Zorich cocycle allow us to derive in section \ref{sec:transferTeich} some results about the Lyapunov spectra of certain transfer cocycles related to the deviation of ergodic averages for generic translation flows. A similar analysis is carried out in the Heisenberg case, based on the results of \cite{FF03}, as outlined in  section \ref{sec:transferheis} of these notes.

\section{Ruelle resonances for pseudo-Anosov diffeomorphisms}
\label{sec:pseudoAnosov}
 
 Let $\Phi:M\to M$ denote a pseudo-Anosov diffeomorphism with orientable stable/unstable foliations and let $\{X,Y\}$ denote the generators
of translation flows along the unstable/stable foliations. There exists $\lambda >1$ such that
$$
\Phi_*(X) = \lambda X  \quad \text{ and } \quad \Phi_*( Y) = \lambda^{-1} Y\,.
$$
Let $\omega$ denote the $\Phi$-invariant area-form.  The form $\omega$ is also invariant for the translation flows of the translation surface given by the pair $\{X,Y\}$.  Let $\Sigma$ denote the set of cone points of the translation
surface. The area form $\omega$ vanishes at $\Sigma$.  Let $\mathcal S_{X,Y}(M)$ denote the space of all smooth functions on $M$, which at a cone $p\in \Sigma$ of order $k \geq 1$, are locally the pull-back of a smooth function under the local branched covering chart $z \to z^{k+1}/(k+1)$ for the translation structure on a neighborhood of~$p$. 

 For any pair $f, g$  of sufficiently smooth complex-valued functions on $M$, we are interested  in the asymptotic for the decay of the correlations 
$$
{\mathcal C}(f,g, n) =\langle f\circ \Phi^n, g \rangle_{L^2(M, \omega)}\,.
$$
Let 
$$\mu_1:=\lambda > \vert \mu_2 \vert \geq  \dots \geq  \vert \mu_{2g-1}\vert  >\mu_{2g} :=\lambda^{-1}$$ 
denote the spectrum $\sigma(\Phi)$ of $\Phi_* $ on $H^1(M, \R)$.  Since $\Phi_*$ is a symplectic map on $H^1(M, \R)$ it follows that
$$
\mu_{2g-i+1}  = \mu_{i}^{-1}\,, \quad \text{ for all } i \in \{2, \dots, 2g-1\}\,. 
$$

The following theorem was recently proved by F.~Faure, S.~Gou\"ezel and E.~Lanneau by methods based on the
analysis of the {\it transfer operator}.

\begin{theorem} The set $\mathcal R$ of Ruelle resonances can be described as follows:
$$
\mathcal R =\{ 1\} \cup  \{ \mu_i \lambda^{-j}\vert  i \in \{2, \dots, 2g-1\} \text{ and } j \in \N\setminus \{0\}\,.\}
$$
All spectral values $\mu_1\lambda^{-j}, \dots, \mu_{2g-1}\lambda^{-j}$ have multiplicity $j \geq 1$. The following
asymptotics holds. For all  functions $f, g\in C_0^\infty (M\setminus \Sigma)$ we have an asymptotic expansion
$$
{\mathcal C}(f,g, n) \approx \sum_{\rho \in \mathcal R} \sum_{i=1}^{I_\rho} c_{\rho, i} (f,g) n^i \rho^n   \,.
$$ 
\end{theorem}

 \begin{definition}  A function $f\in L^2(M, \omega)$ is an iterated coboundary of order $k\geq 1$ with  transfer function $u \in L^2(M, \omega)$ if
$u$ is a weak solution of the  equation 
$$
X^k u = f  \,.
$$
\end{definition} 

\begin{lemma} 
 For any $k$-iterated coboundary $f \in L^2(M,\omega)$ with transfer function $u \in L^2(M, \omega)$ and for
any $g \in L^2(M, \omega)$ such that $X^k g \in L^2(M, \omega)$ we have the estimate
$$
\vert {\mathcal C}(f,g, n) \vert  \leq   \lambda^{-kn} \vert u \vert_0  \vert   X^k g \vert_0\,.
$$
\end{lemma}

\begin{proof} The statement follows immediately from the following identities:
$$
\begin{aligned}
\langle (X^k u)\circ \Phi^n, g \rangle_{L^2(M, \omega)} &= 
\lambda^{-kn} \langle \Phi^n_\ast(X)^k (u) \circ \Phi^n, g \rangle_{L^2(M, \omega)} 
 \\ & =\lambda^{-kn} \langle X^k (u\circ \Phi^n), g \rangle_{L^2(M, \omega)} 
\\ &=   (-1)^k  \lambda^{-kn} \langle u\circ \Phi^n, X^k g \rangle_{L^2(M, \omega)} \,.
\end{aligned}
$$

\end{proof}

In view of the lemma, a natural question is how to characterize  iterated coboundaries for the unstable (or stable) vector field.
Coboundaries were characterized in \cite{F97}, \cite{F07} for generic translation flows, and in \cite{MMY05} also in case of invariant
foliations of pseudo-Anosov diffeomorphisms.

\smallskip
For all $s\in \R$, let $W^s_{X,Y}(M)$ denote the $L^2$ Sobolev spaces of the translation surface $\{X,Y\}$
(see \cite{F97} and \cite{F02} for definitions and basic properties of these spaces for integer exponent,  and \cite{F07}, for the subtler  case of real, non-integer exponent).

\begin{theorem}  There exists a finite dimensional space ${\mathcal I}^{-s}_X(M)  \subset W^{-s}_{X,Y}(M)$ of 
$X$-invariant distributions  such that for any $f \in W^s_{X,Y}(M)$   (with $s >s_0$) such that
$$
D(f) =0 \, , \quad \text{ for all } \,D\in  {\mathcal I}^{-s}_X(M)\,,
$$
is an $X$-coboundary with zero-average transfer function $u \in W^t_{X,Y}(M)$ for all $t~<~s~-~s_0$. In addition there exists a constant $C_{s,t}>0$ such that for all $t<s-s_0$, 
$$
\vert u\vert_t \leq   C_{s,t}   \vert  f \vert_s \,.
$$
\end{theorem} 

\begin{remark}  The loss of derivatives in Sobolev spaces  was estimated carefully in \cite{F07}
to be $3+$ for almost all directions on any given translation surface, and $1+$ under the assumption of
hyperbolicity of the Kontsevich--Zorich renormalization cocycle, for almost all translation flows. 
In the H\"older class a loss of $1+\delta(s)$ (with $\delta(s) \to 0^+$ for $s\to 1+$) was proved by
Marmi and Yoccoz \cite{MY16} under similar hypotheses. Finally, the work of Faure, Gou\"ezel  and Lanneau 
should lead to a H\"older loss of $1+$  in the ``periodic'' (pseudo-Anosov) case. However, their results are explicitly stated only for spaces with integer exponents. 
\end{remark}

Let us recall that  $\mathcal S_{X,Y}(M)$ denotes the space of all smooth functions on $M$, which at 
a cone $p\in \Sigma$ of order $k \geq 1$, are locally the pull-back of a smooth function under the local
branched covering chart $z \to z^{k+1}/(k+1)$ for the translation structure on a neighborhood
of $p$. The dual space $\mathcal S'_{X,Y}(M)$  is called the space of {\it tempered currents} for
the translation structure. 

Let ${\mathcal I}_X \subset \mathcal S'_{X,Y}(M)$ denote the space of all {\it tempered} 
$X$-invariant distributions. There exists a map $C: {\mathcal I}_X  \to \mathcal Z (M)$ into the space 
$\mathcal Z (M)$ of all closed $1$-currents on $M$, defined as
$$
C(D) :=  D \cdot \imath_X  \omega \,.
$$
The range of the map $C:{\mathcal I}_X(M)  \to \mathcal Z(M)$ is the subspace $\mathcal B_X(M)$ of {\it tempered 
basic} currents for the unstable foliation of the pseudo-Anosov map $\Phi$, that is, the subspace of currents $C$ such that
$$
L_X C = \imath_X C =0 
$$
($L_X$ denotes the operator of Lie derivative and $\imath_X$ the contraction on currents).

\smallskip
The de Rham cohomology map  $\mathcal R: \mathcal Z(M) \to H^1(M, \R)$, restricted to the subspace 
$\mathcal B_X(M)$ of basic currents has range
$$
H^1_X(M, \R):= \{ c  \in H^1(M, \R) \vert   c \wedge \imath_X  \omega =0\} \,.
$$

Let now $\mathcal I^{-s}_X(M) = \mathcal I_X(M) \cap W^{-s}_{X,Y}(M)$ denote the subspace of invariant distributions of finite order $s>0$ and let $\mathcal B^{-s}_X(M)$ denote the corresponding space of basic currents.

\begin{lemma} \cite{F02}
For any $C\in \mathcal B^{-s}_X(M)$ such that $[C]=0$ in $H^1(M,\R)$,  there exists $C' \in  \mathcal B^{-s+1}(M)$ such that $C=  \mathcal L_Y C'$.
\end{lemma}
\begin{proof} By the de Rham theorem there exists $U \in W^{-s+1}_{X,Y}(M)$ such that $dU =C$. We have to prove that $C'= \imath_X U \in \mathcal B_X(M)$ and that $ \mathcal L_Y C' =C$.  We have
$$
d \mathcal L_X U = d \imath_X d U =  \mathcal L_X dU =  \mathcal L_X C =0\,,
$$
which implies $\mathcal L_X U$ is constant, hence it vanishes. Thus $U \in \mathcal I^{-s+1}_X(M)$
and we have
$$
dC' = d \imath_X U = \mathcal L_X U - \imath_X dU = -\imath_X C =0\,.
$$
Since $C'$ is closed and $\imath_X C'=0$, it follows that $C' \in  \mathcal B_X(M)$.
\end{proof}

\begin{lemma} The spectrum of $\Phi_*$ on the space $\mathcal B_X(M)$ of basic currents is
$$
\sigma_{\mathcal B_X(M)} (\Phi_*) := \{\lambda\} \cup  \{ \mu_i \lambda^{-j}  \vert  i \in \{2, \dots, 2g-1\} \text{ and } j\in \N\}\,.
$$
Consequently, the spectrum of $\Phi_*$ on the space $\mathcal I_X(M)$ of invariant distributions is
$$
\sigma_{\mathcal I_X(M)} (\Phi_*) := \{1\} \cup  \{ \mu_i \lambda^{-j-1}  \vert  i \in \{2, \dots, 2g-1\} \text{ and } j\in \N\}\,.
$$
\end{lemma} 
\begin{proof}
Let $\mathcal L_Y :  \mathcal  B_X(M)   \to \mathcal B_X(M)$ denote the Lie derivative operator
with respect to the vector field $Y$ on $M$.  By the previous lemma, the maps $\mathcal L_Y: 
\mathcal B_X(M) \to   \mathcal B_X(M)$ and the de Rham cohomology map $\mathcal R:\mathcal B_X(M) \to H^1_X(M,\R)$ give a $\Phi_*$ equivariant exact sequence
$$
0\to \C \imath_X\omega  \to \mathcal B^{-s+1}_X(M) \to   \mathcal B^{-s}_X(M) \to  H^1_X(M,\C) \to  0  \,.
$$
Finally, for all $C\in \mathcal B_X(M)$,  we have, for all $i, j \in \N$, 
$$
(\Phi_* -\mu I)^i \mathcal L_Y^j (C) = \lambda^{-ij} \mathcal L_Y^j  (\Phi_* -  \lambda^{j} \mu I)^i  (C)\,.
$$
Finally, as explained above, the map $C:  \mathcal I_X(M) \to  \mathcal B_X(M)$ defined as
$$
C (D)= D \imath_X \omega \,, \quad \text{ for all } D \in \mathcal I_X(M) \,,
$$
is an isomorphism. We clearly have
$$
C ( \Phi_* D ) =  (\Phi_* D ) \imath_X \omega = \lambda^{-1} \Phi_* (D \imath_X \omega) = 
\lambda^{-1}  \Phi_*(C(D))\,,
$$
hence the proof is completed.
\end{proof}

 Obstructions on functions to be iterated coboundaries can be constructed as follows. A coboundary $f\in W^s_{X,Y}(M)$ (with $s> 2s_0$)  with smooth transfer function $u \in W^t_{X,Y}(M)$ (with $s-s_0>t>s_0$) is a $2$-iterated coboundary with $L^2$ transfer function   if $u$ is itself a coboundary with smooth transfer function if and only if 
 $$
 D(u)=0 \,, \quad \text{ for all } D\in \mathcal I^{-t}_X(M) \,.
 $$
 We can therefore define a linear functional as follows.  Let 
 $$G_X : \text{Ker}\,\mathcal I^{-s}_X(M)  \to  W^t_{X,Y}(M)$$ 
 denote the Green operator such that $u:=G_X(f)$ is the zero-average solution of the cohomological equation $Xu=f$. For every  $D\in \mathcal I^t_X(M)$, we define
  $$
 D' (f) =  D( G_X (f) ) \,, \quad \text{ for all  } f \in \text{Ker}\,\mathcal I^{-s}_X(M) \subset W^s_{X,Y}(M)\,.
 $$
 The above functional can be extended to the whole space $W^s_{X,Y}(M)$ as zero on the orthogonal complement of $\text{Ker}\,\mathcal I^{-s}_X(M)$.  In fact, it follows from the above results that if $D \in W^{-t}_{X,Y}(M)$, then $D' \in W^{-s}_{X,Y}(M)$ for any $s-t>s_0$.  In fact, for $f\in \text{Ker}\,\mathcal I^{-s}_X(M)$,
 $$
 \vert D' (f) \vert = \vert D( G_X (f) )\vert  \leq \vert D\vert_{-t} \vert G_X (f) \vert_t \leq C_{s,t} \vert f \vert_s \,.
 $$
 This construction can be iterated.  Another point of view is based on the remark that
 $$
 XD' (f) = D'(Xf) = D( G_X (Xf) ) =D(f) \,, \quad \text{for all } f \in  W_{X,Y}^\infty(M)\,.
 $$
 It follows that $D'$ can be defined as a distributional solution $D'$ of the equation
 $$
 XD'=D \,.
 $$
 The solution of the above equation is unique up to the addition of $X$-invariant distributions, and we can define
$D' \in W^{-s}_{X,Y}(M)$ as the unique solution orthogonal to the subspace $\mathcal I^{-s}_X(M)$ of invariant
distributions. 

\smallskip
For all $k\in \N$, let ${\mathcal I}_{X,k} ^{-s}(M) \subset W^{-s}_{X,Y} (M)$ denote the subspace
$$
{\mathcal I}_{X,k} ^{-s}(M):=\{ D\in W^{-s}_{X,Y} (M) \vert  X^{k} D=0\}\,.
$$
We have the following results on the iterated cohomological equation
 
 \begin{lemma}  Any function $f \in W^{s}_{X, Y} (M)$ of zero average with $s > ks_0+t$ is a $k$-iterated coboundary 
 with transfer function $u \in W^t_{X,Y}(M)$  if
 $$
 D (f) =0 \,, \quad \text{ for all }   D\in \mathcal I_{X,k}^{-s}(M)  \,.
 $$
In addition, there exists a constant $C^{(k)}_{s,t}>0$ such that 
$$
\vert u\vert_t \leq  C^{(k)}_{s,t} \vert f \vert_s\,.
$$
 \end{lemma} 
 \begin{proof}
 We argue by  induction on $k\in \N\setminus \{0\}$. For $k=1$ we have the result on solutions of the
 cohomological equation.  Let us assume that the result  holds for $k\in \N\setminus\{0\}$. Let $u_k
 \in W^r_{X,Y} (M)$ (with $r < s-ks_0$)  be the zero-average solution of the iterated cohomological equation 
 $X^k u_k=f$ for a function $f \in \text{Ker} \,\mathcal I_{X,k+1}^{-s}(M)$.
 
 We claim that for all $D\in \mathcal I^{-r}_X(M)$ we have $D(u_k)=0$. By definition,  the set $\mathcal I_X^{-r} (M) \subset  X^k \mathcal I_{X,k+1}^{-s}(M)$, and for all $D_{k+1} \in \mathcal I_{X,k+1}^{-s}(M)$, we have
 $$
 X^k D_{k+1} (u_k) = (-1)^k D_{k+1} (X^k u_k) =  (-1)^k D_{k+1} (f) =0\,,
 $$
hence the claim is proved. By the result on the cohomological equation, there exists 
$u:= u_{k+1} \in  W^{t}_{X,Y} (M)$  (with $t < r-s_0$) such that $Xu_{k+1} = u_k$. 
The argument is complete except for the a priori bounds. We have
$$
\vert u\vert_t \leq  C_{r,t} \vert u_k \vert_r \leq C_{r,t} C^{(k)}_{s,r}  \vert f \vert_s\,.
$$ 
 \end{proof}
 
 Let $\mathcal I_{X,\infty} (M)$ denote the following distributional space 
 $$
 \mathcal I_{X,\infty} (M) := \bigcup_{k\in \N\setminus \{0\}}  \mathcal I_{X,k} (M)
 = \bigcup_{k\in \N\setminus \{0\}} \{D \in \mathcal S_{X,Y}(M) \vert  X^k D=0\} \,.
 $$
  By definition we have the following inclusions:
 $$\mathcal L_X  {\mathcal I}_{X,k} ^{-s}(M) =  {\mathcal I}_{X,k-1} ^{-s-1}(M)
 \quad \text{ and } \quad \mathcal L_Y  {\mathcal I}_{X,k} ^{-s}(M) \subset 
   {\mathcal I}_{X,k-1} ^{-s-1}(M)\,.$$
   
The operators Lie derivative operators $\mathcal L_X$ and
 $\mathcal L_Y$ are a creation and an annihilation operators for the spectrum
 of $\Phi_*$  on $\mathcal I_{X,\infty} (M)$, in fact
 $$
 \Phi_* \circ \mathcal L_X =    \lambda \mathcal L_X  \circ  \Phi_* \quad \text{ and } \quad
  \Phi_* \circ \mathcal L_Y =    \lambda^{-1}  \mathcal L_Y \circ  \Phi_*\,.
 $$
By the above description of the spectrum  of $\Phi_*$ on the space $\mathcal I_X(M)$ of invariant distributions,
we derive the following
 \begin{lemma}   For every $k\in \N\setminus\{0\}$, the spectrum of $\Phi_*$ on the space $\mathcal I_{X, k}(M)$ 
of generalized invariant distributions is  the set 
 $$
\sigma_{\mathcal I_{X,k}(M)} (\Phi_*) := \{1\} \cup  \{ \mu_i \lambda^{-j}  \vert  i \in \{2, \dots, 2g-1\} \text{ and } j\in 
\{1, \dots, k\}\}\,.
$$
Consequently, the spectrum of $\Phi_*$ on the space $\mathcal I_{X, \infty}(M)$ 
 is  the set 
$$
\sigma_{\mathcal I_{X,\infty}(M)} (\Phi_*) := \{1\} \cup  \{ \mu_i \lambda^{-j}  \vert  i \in \{2, \dots, 2g-1\} \text{ and } j\in 
\N\setminus\{0\}\}\,.
$$
All spectral values $\mu_2 \lambda^{-l}, \dots, \mu_{2g-1} \lambda^{-l} $ have multiplicity exactly equal to $l$.

 \end{lemma}
\begin{proof} 
We have proved that the spectrum of $\Phi_*$ on the space $\mathcal I_X(M)$ is the above set.
Since $L_X: \mathcal I_{X,2}(M)\to \mathcal I_{X,1}(M)=\mathcal I_X(M)$ is surjective with kernel
$\C \omega$ and acts as a ``creation'' operator on the spectrum, it follows that the spectrum of $\Phi_*$ on
$ \mathcal I_{X,2}(M)/ \mathcal I_{X,1}(M)$ is the set
$$
\{ \mu_i \lambda^{-j-2}  \vert  i \in \{2, \dots, 2g-1\} \text{ and } j\in \N\}
$$
By induction we can prove that the spectrum of $\Phi_*$ on $ \mathcal I_{X,k+1}(M)/ \mathcal I_{X,k}(M)$
is the set
$$
\{ \mu_i \lambda^{-j-(k+1)}  \vert  i \in \{2, \dots, 2g-1\} \text{ and } js\in \N\}\,.
$$
It follows that the multiplicity of a spectral value $\mu \lambda^{-l-1}$ equals the set of integer
solutions of the equation $j+k = l+1$ with $j\in \N$ and $k\in \N\setminus\{0\}$, that is, it is equal
to $l+1$. In other terms, given an eigenvector $D \in \mathcal I_{X,\infty}(M)$ with eigenvalue
$\mu \lambda^{-l-1}$ a basis of the corresponding eigenspace can be  written as
$$
\{ D, \mathcal L_X\mathcal L_Y D, \dots, (\mathcal L_X\mathcal L_Y)^{l} D\}\,.
$$
In fact, since $\Phi_* (D) = \mu \lambda^{-l-1} D$, the identity
$$
\Phi_* (\mathcal L_X^{l+1} D ) = \lambda^{l+1} \mathcal L_X^{l+1} \Phi_*( D ) = \mu \mathcal L_X^{l+1} D 
$$
implies (since $\mu$ does not belong to the spectrum of $\Phi_*$ on $\mathcal I_{X,\infty}(M)$),  that $\mathcal L_X^{l+1}D=0$, hence $(\mathcal L_X\mathcal L_Y)^{l+1} D = \mathcal L_Y^{l+1} \mathcal L^{l+1}_XD =0$.
It can be verified that the above system is indeed linearly independent, hence it is a basis of the eigenspace.
\end{proof}

\medskip
 From the above analysis  we can derive an asymptotic for the correlations as follows. For every $g \in W^r_{X, Y}(M)$ let  
 ${\mathcal C}_g \in W^{-s}_{X,Y}(M)$ defined as 
 $$
 {\mathcal C}_g (f) = \langle f, g \rangle _{L^2(M, \omega)}\,.
 $$
 By definition it follows that 
 $$
 \Phi_*^n (  {\mathcal C}_g ) (f) = {\mathcal C}_g (f\circ \Phi^n) = {\mathcal C} (f,g, n)\,.
 $$
  We recall that for every $s>k s_0$, every function $f \in \text{Ker} \,{\mathcal I}^{-s}_{X, k} (M) \subset W^s_{X,Y} (M)$   is a $k$-iterated coboundary  with transfer function 
 $u\in W^t_{X,Y}(M)$ for all $t< s-ks_0$. 
 
Let ${\mathcal R}^{-s}_k \subset  \mathcal I^{-s}_{X, k} (M) $ denote a set of (generalized) eigenvectors (Ruelle eigenstates) for the linear action of $\Phi_*$ on  $\mathcal I^{-s}_{X, k} (M)$.
The distribution $ {\mathcal C}_g $ can then be expanded as
 $$
  \Phi_*^n( {\mathcal C}_g)  = \sum_{ D\in {\mathcal R}^{-s}_k }   c^{(n)}_D(g)   D  +   R^{(n)}_g \,,
 $$
 with remainder distributions $R^{(n)}_g  \in  \mathcal I^{-s}_{X, k} (M)^\perp \subset W^{-s}_{X,Y}(M)$. 
 
 \begin{lemma} For any $s>k s_0$, there exists a constant $C_s>0$ such that, for all $g \in L^2(M, \omega)$ with  $X^kg\in  L^2(M, \omega)$
 and for all $n\in \N$, we have 
 $$
  \vert R^{(n)}_g \vert_{-s} \leq  C_s  \vert X^k g\vert_0 \lambda^{-kn}
 $$
 
 \end{lemma}
 \begin{proof} For any $f \in W^s_{X,Y}(M)$ we have the orthogonal decomposition
 $$
 f= f_0 + f_1  \quad \text{ with }   f_0 \in  [\text{ker} ({\mathcal R}^{-s}_k)] ^\perp \text{ and } f_1 \in \text{ker} ({\mathcal R}^{-s}_k)\,.
 $$
 Since by construction $R_g  \in  ({\mathcal R}^{-s}_k)^\perp$ and $f_0 \in \text{ker}[ ({\mathcal R}^{-s}_k)^\perp]$,  it follows that
 $$
 R^{(n)}_g (f) =  R^{(n)}_g (f_0 + f_1) =  R^{(n)}_g (f_1) =   [\Phi_*^n ({\mathcal C}_g)] (f_1)\,.
 $$
 Since the function $f_1 \in \text{ker} ({\mathcal R}^{-s}_k)$, it is a $k$-iterated coboundary with transfer function
 $u_1\in L^2(M, \omega)$. It follows that
 $$
 \vert R^{(n)}_g (f)\vert \leq  \vert {\mathcal C} (f_1, g , n) \vert \leq  \lambda ^{-kn} \vert u_1\vert_0 \vert X^k g\vert_0 \leq C_s  \lambda ^{-kn} \vert X^k g\vert_0 
 \vert f_1\vert_s \,.
 $$
 Since by orthogonality $\vert f_1 \vert_s \leq \vert f \vert_s$ we finally derive the stated bound.
  \end{proof} 
 
\begin{lemma}  For any $D\in \mathcal R^{-s}_k$, let  $\lambda_D= \mu \lambda^{-h} \in \C$, with $\mu\in \{\mu_2, \dots, \mu_{2g}\}$, denote the corresponding eigenvalue. If $\vert \lambda^{k-h}\mu \vert >1$, there exist
$c_D(g) \in \C$ and $C_D(g) >0$ such that 
$$
 \vert c^{(n)} _D(g)  - c_D(g) \lambda^{n}_D  \vert  \leq  C_D(g) (\lambda^{k-h} \mu)^{-n} \,.   
$$
\begin{proof} By definition we have
$$
\begin{aligned}
 \Phi_*^{n+1}( {\mathcal C}_g)  &= \sum_{ D\in {\mathcal R}^{-s}_k }   c^{(n+1)}_D(g)   D  +   R^{(n+1)}_g 
 = \sum_{ D\in {\mathcal R}^{-s}_k }   c^{(n)}_D(g)    \Phi_* (D)   +   \Phi_*(R^{(n)}_g ) \\
 &= \sum_{ D\in {\mathcal R}^{-s}_k }   c^{(n)}_D(g)  \lambda_D  D   +   \Phi_*(R^{(n)}_g )  \,.
\end{aligned}
$$
It follows by the above identities and the previous  lemma that,  for every $n\in \N$, there exists $r_n\in \C$ with 
$\vert r_n(g) \vert \leq C_s(g) \lambda^{-k n}$  such that
$$
c^{(n+1)}_D(g)  =   \lambda_D  c^{(n)}_D(g)  + r_n(g) \,.
$$
By solving the difference equation we can write
$$
c^{(n)}_D(g) = \lambda_D^n  \left( c^{(0)}_D(g) + \sum_{l=0}^{n-1}  \lambda_D ^{-l-1}  r_l(g)  \right) \,.
$$
By the estimate on the remainder term we have
$$
\vert \lambda_D ^{-l}  r_l(g) \vert \leq  C_s(g)  [\lambda_D \lambda^k ]^{-l}  \,,
$$
and, since $\vert \lambda_D \lambda^k\vert  = \vert \mu  \lambda^{-h}  \lambda^k\vert = \vert  \lambda^{k-h} \mu\vert >1$, 
the series in the above formula is a convergent geometric series, hence the statement follows.

\end{proof} 

\begin{remark}
There is a symmetry, for all $f, g \in L^2(M, \omega)$ and all $n\in \N$, 
$$
\langle f \circ \Phi^n, g\rangle_{L^2(M, \omega)}  =  \langle f, g\circ \Phi^{-n} \rangle_{L^2(M, \omega)}\,.
$$
It follows that in the above expansion the coefficients $c_D$ are given by generalized invariant distributions
for the stable translation flow $Y$ on $M$.  In other terms, for all coefficients $c(f,g)$ in the Ruelle-type 
expansion of correlations there exists $k\in \N\setminus \{0\}$ such that, for all $f, g \in \mathcal S_{X,Y}(M)$, 
$$
c(Y^k f, g) = c (f, X^k g) =0 \,.
$$
\end{remark}
This duality is related to that discovered by Bufetov \cite{Bu14} in his work on limit distributions of ergodic averages
(see also \cite{BuF14} for horocycle flows.)

\end{lemma}

\begin{problem}  Generalize the results of Faure, Gou\"ezel and Lenneau \cite{FGL} to the case of non-linear
pseudo-Anosov maps of surfaces, in particular to smooth  perturbations of pseudo-Anosov maps by 
sufficiently small perturbations supported on the complement of the singularity set. 
\end{problem}
The author has proved in \cite{F} that, assuming that the non-linear pseudo-Anosov map has a Margulis
measure, then all the Ruelle resonances in the interval $[e^{h_{top}}, 1)$ are determined by the action of the
map on the first cohomology (or homology) of the surface.

\section{Transfer cocycles and generic translation flows}
\label{sec:transferTeich}

Let $\mathcal H_g$ denote the space of Abelian holomorphic differentials ($1$-forms) on
Riemann surfaces on a topological (smooth) surface $S$ of genus $g\geq 2$.  We recall 
that there is a natural identification between  translation structures $\{X,Y\}$, given by pairs
of commuting transverse vector fields with appropriate normal forms at the singularities, and 
Abelian holomorphic differentials $h \in \mathcal H_g$ on Riemann surfaces.

For any matrix $A\in GL(2, \R)$ and $h\in \mathcal H_g$ and let $Ah$ denote the closed complex-valued $1$-form
$$
A h =  [a \text{\rm Re} (h) + b \text{\rm Im} (h)] + \imath [c \text{\rm Re} (h) + d \text{\rm Im} (h) ] \,,\quad   \text{\rm for } \,A:=\begin{pmatrix} a & b \\ c & d  \end{pmatrix}  \,.
$$
Alternatively, in the language of pairs of (commuting) vector fields $\{X,Y\}$ we can define the $SL(2, \R)$ action as follows:
$$
A \begin{pmatrix} X \\ Y \end{pmatrix} = \begin{pmatrix} aX +bY \\ cX+ dY \end{pmatrix} \,,\quad   \text{\rm for } \,A:=\begin{pmatrix} a & b \\ c & d  \end{pmatrix}  \,.
$$

There exists a unique complex structure on $S$, that is, a unique Riemann surface $M_{Ah}$, such that $Ah$ is holomorphic on $M_{Ah}$. This is the well-known standard definition of the action of the group $GL(2, \R)$ on ${\mathcal H}_g$. 

This action of $GL(2, \R)$ commutes with the diagonal action of the group $\text{\rm Diff}^+(S)$ of orientation preserving diffeomorphisms of the surface $S$, hence it induces an action on the quotient  space $\hat {\mathcal H}_g :=(\mathcal H_g \times S)/ \text{\rm Diff}^+(S)$:
$$
A [h] =   [Ah]\,,  \quad \text{ for all } (A,[h])  \in  GL(2, \R) \times {\mathcal H}_g .
$$
The subaction of the diagonal subgroup $(g_t) <SL(2, \R)$ on ${\mathcal H}_g$ is known as the {\it Teichm\"uller geodesic flow}, the actions of the unipotent subgroups of $SL(2, \R)$ are known as the {\it Teichm\"uller horocycle flows}.

\smallskip
We introduce the notion of a {\it transfer cocycle}. For every $h\in {\mathcal H}_g$, let ${\mathcal S}'_h(M)$ denote a space of tempered distributions (or currents) on $M$ defined in terms of the translation structure  determined by $h$ on the underlying Riemann surface $M$. Let  ${\mathcal S}'_g(M)$ denote the bundle
$$
{\mathcal S}'_g(M) = \{ (h, D) \vert   D \in {\mathcal S}'_h(M) \}\,.
$$
It is possible to extend the $GL(2, \R)$ action on $\mathcal H_g$ to the bundle ${\mathcal S}'_g(M)$ by parallel transport with respect to the trivial connection, that is,
$$
A(h,D) =  (Ah, D) \,, \quad \text{ for all } (h,D) \in {\mathcal S}'_g(M)\,.
$$
Since the above action of $GL(2, \R)$ commutes with the diagonal action of $\text{\rm Diff}^+(S)$ on ${\mathcal S}'_g$ defined as
$$
\phi (h,D) = (\phi^*(h), \phi_* (D))   \,, \quad \text{ for all } (h,D) \in {\mathcal S}'_g(M)\,,
$$
it follows that the action of $GL(2, \R)$ on ${\mathcal S}'_g$ passes to the quotient to an action on the vector bundle ${\mathcal S}'_g := {\mathcal S}'_g / \text{\rm Diff}^+(S)$,
over the action of $GL(2, \R)$ on ${\mathcal S}'_g$. The subaction of the of the diagonal subgroup $(g_t) <SL(2, \R)$ on $\mathcal H_g$ is by definition
the transfer cocycle $(\mathcal L_t)$ over the Teichm\"uller flow

\smallskip
The above construction gives the Kontsevich--Zorich cocycle (a finite dimensional cocycle) when  
${\mathcal S}'_h (M)$ is replaced with the cohomology $H^1(S, \R)$  or    $H^1(S, \Sigma, \R)$ (or their complexifications) for all $h \in \mathcal H_g$ . We are interested in the case when ${\mathcal S}'_g(M)$ is a Banach (Hilbert) bundle  of distributions or (closed) currents.  The general question is whether the cocycle has well-defined Lyapunov exponents and whether it has a spectral gap. In fact, we are interested in the Oseledets decomposition of particular distributions, or currents, such as those given by rectifiable arcs or by ''correlations''.

\medskip
For instance, our generalization of the Ruelle resonance problem to this setting is as follows. For any function $g\in L^2(S, \omega_h)$ on $S$ we can consider the distribution $\mathcal C_h(g) \in W^{-s}_h(S)$  defined as
$$
\mathcal C_h(g) (f )  =  \langle f, g \rangle_{L^2(S,\omega_h)}\,.
$$
The distribution $\mathcal C_h(g)$ is in fact an absolutely continuous measure, but we ask what is its asymptotic behavior, for instance under the  Teichm\"uller flow (that is, under the action of the distributional cocycle 
$(\mathcal L_t)$ defined above). 

In \cite{F02} the author considered the above construction with ${\mathcal S}'_g (M)$ equal to the bundle ${\mathcal Z}^{-1}_g(M)$ with fiber the space of closed $1$-currents in the Sobolev space $W^{-1}_h(M)$.  It was proved there that the transfer 
cocycle has Lyapunov exponents equal to the Kontsevich--Zorich exponents with respect to any KZ-hyperbolic $SL(2,\R)$-invariant measure (with almost any fiber generated by basic currents for the horizontal and 
vertical vector fields). Results on the deviation of ergodic averages were then derived by a short argument based 
on the de Rham theorem.

 \medskip
 We outline this argument below.  Let $ \mathcal Z^{-1}_g (M)$ and $\mathcal E^{-1}_g(M)$ denote respectively
 the bundles of closed and exact currents with coefficients in the Sobolev bundle $W^{-1}_g(M)$ with fiber 
 at every $h=(X,Y) \in \mathcal H_g$ the dual Sobolev space $W_{X,Y}^{-1} (M)$. 
 
 \smallskip
 Under the hypothesis that the Kontsevich--Zorich cocycle is non-uniformly hyperbolic, by a representation theorem, for almost all $h=(X,Y)\in \mathcal H_g$, we have
 $$
 {\mathcal Z}^{-1}_h (M) := \mathcal B^{-1}_X (M) \oplus \mathcal B^{-1}_Y (M) \oplus \mathcal E^{-1}_h(M)\,.
 $$
 Let $\mu$ be a Borel probability measure, invariant under the Teichm\"uller flow and let
 $$
 1=\lambda^\mu_1 > \lambda^\mu_2 \geq \dots \geq \lambda^\mu_g (\geq 0) \geq -\lambda^\mu_g \geq \dots\geq 
 -\lambda^\mu_1=-1
 $$
 denote the Kontsevich--Zorich exponents (we recall that $\lambda_2^\mu<1$ was proved by W.~Veech \cite{V86} for some class of measures, and in \cite{F02} in general; that $\lambda_g>0$ was proved in \cite{F02}, \cite{F11} for the canonical measures and for other $SL(2,\R)$-invariant measures; the simplicity of the spectrum was proved by Avila and Viana \cite{AV07} for the canonical measures). 
 
 \begin{theorem} \cite{F02} Let us assume that the Kontsevich-Zorich cocycle is non-uniformly hyperbolic. The transfer
 cocycle on the bundle  ${\mathcal Z}^{-1}_g(M)$ has Lyapunov spectrum 
 $$
 1=\lambda^\mu_1 > \lambda^\mu_2 \geq \dots \geq \lambda^\mu_g \geq 0 \geq -\lambda^\mu_g \geq \dots\geq 
 -\lambda^\mu_1=-1\,.
 $$
 The Oseledets sub-bundle of the exponent $0$ is the sub-bundle of exact currents (which is infinite dimensional).
 \end{theorem}
 
 Since arc of orbits of the translation flows are, as currents, at a distance from closed currents bounded
 by the diameter of the translation surface, it follows that their behavior under the transfer cocycle is described
 by the above Lyapunov spectrum. 
 \begin{corollary} For all $i\in \{1, \dots, g\}$, let $D^h_{i, 1}, \dots, D^h_{i, m_i}$  denote a system of invariant distributions associated to the basic currents in the Oseledets space for the exponents $\lambda^\mu_i >0$. 
 For all functions  $f\in W^1_{X,Y}(M)$ such that $D_{j, 1} (f)= \dots= D_{j,m_j}(f)=0$, for all $1\leq j\leq i<g$, we have
 $$
\limsup_{T\to +\infty}  \frac{1}{\log T}  \log  \vert  \int_0^T  f\circ \Phi^X_t (x) dt  \vert   \leq  \lambda_{i+1}\,,
 $$ 
 and equality holds if there exists a distribution $D^h_{i+1}$ with basic currents of Lyapunov exponents 
 $\lambda_{i+1}\geq 0$ such that $D^h_{i+1}(f) \not =0$.  If $D_{j, 1} (f)= \dots= D_{j,m_j}(f)=0$, 
 for all $1\leq j\leq i\leq g$, then 
 $$
\limsup_{T\to +\infty}  \frac{1}{\log T}  \log  \vert  \int_0^T  f\circ \Phi^X_t (x) dt  \vert  =0\,.
 $$ 
 
 \end{corollary} 
 
 After recalling the above results on ergodic averages, we return to the generalization of Ruelle eigenstates
 to the generic translation flows. By the results on solutions of the cohomological equation for almost all
 translation flows (see \cite{F97}, \cite{F07}, \cite{MMY05}, \cite{MY16}), it is possible to construct for almost
 all translation surface $\{X,Y\}$ the space of generalized (iterated) invariant distributions $\mathcal I_{X,k}(M)$
 (for all $k\in \N\setminus\{0\}$) and $\mathcal I_{X,\infty}(M)$ and the corresponding bundles $\mathcal I_k(M)$
 and $\mathcal I_\infty(M)$ over the  moduli space.  We have the following result
 
 \begin{theorem} \label{thm:transfer_cocycle} Let us assume that the KZ cocycle is non-uniformly hyperbolic. The Oseledets spectrum
 of the transfer cocycle over the the bundle $\mathcal I_\infty(M)$ has exponents
 $$
 \{1\} \cup \{ \pm \lambda_i - j   \vert  i=2, \dots, g,  j \in \N\setminus\{0\} \} \,.
 $$
 The Lyapunov exponent $1$ is simple and corresponds to the subbundle of $\mathcal I_\infty (M)$ given by the invariant area on $M$. In addition, the Lyapunov exponent $\pm \lambda_i -j  $ has multiplicity exactly $j$, for all $i\in \{2, \dots, g\}$ and $ j \in \N\setminus\{0\}$.
  \end{theorem} 
 
 \begin{problem} Generalize the results of Faure, Gou\"ezel and Lanneau \cite{FGL} on the deviation of ergodic
 averages and cohomological equations for the unstable foliations of pseudo-Anosov maps to (measure) generic
 translation flows on higher genus surfaces. 
 \end{problem}

 It should be possible to prove a version the above result on the basis of the general  analytic techniques
 pionereed by Giulietti and Liverani \cite{GL} and exploited in the work of Faure, Gou\"ezel and Lanneau. 
 In particular, for every translation surface $h=(X,Y)$, let $\mathcal S'_h$ denote a distributional space
 introduced in the work of Faure, Gou\"ezel and Lanneau. These spaces are a straightforward adaptation of
 to translation surfaces of the anisotropic Banach spaces introduce by S.~Gou\"ezel and C.~Liverani \cite{GouL06}. 
 It should be possible to prove that the corresponding transfer cocycle  is a linear cocycle of {\it quasi-compact} operators on a Banach bundle. (We recall that a quasi-compact operator is an operator equal to a sum of a compact operator and an operator with ``small'' spectral radius). For quasi-compact cocycles on Banach bundles, the {\it multiplicative ergodic theorem} holds (see for instance~\cite{GQ15} or~\cite{Bl16} and references therein), hence we would derive that the transfer cocycle has a Lyapunov spectrum. The main remaining question would concern the (generic) multiplicity of the top Lyapunov exponent. It should be possible to derive from the generic minimality, or (unique) ergodicity, of generic translation flows that such an exponent is simple. By this approach it should be possible to give a new proof of a spectral gap result related to the polynomial bounds on ergodic integrals found by Kontsevich--Zorich \cite{KZ97}, \cite{Zo97}. 
 
 The Lyapunov-Ruelle spectrum could be computed in terms of Lyapunov exponent of the Kontsevich--Zorich cocycle on the cohomology bundle only by attaching cohomology classes
 to distributions in the stable or unstable subbundles of the transfer cocycle, as in the computation  of Ruelle resonances 
 of pseudo-Anosov maps, explained in these notes, or in the entirely analogous proof of the above Theorem~\ref{thm:transfer_cocycle}, which we have omitted.

 \section{Ruelle resonances for geodesic flows in constant negative curvature}
 \label{sec:horo}
 Let $\phi^X_\R:M\to M$ denote a geodesic flow on the unit tangent bundle $T_1(S)$ of a finite-volume hyperbolic surface. The $3$-manifold $T_1(S)$ can be identified to a quotient $M:=\Gamma \backslash SL(2, \R)$ of the
 group $SL(2, \R)$ with respect to a lattice $\Gamma < SL(2, \R)$.

Let   $\{X,U,V\}$ denote the frame of $SL(2, \R)$ such that $X$ is the generator of the horocycle flow, $U$ is
the generator of the instable horocycle flow, and $V$ of the stable horocycle flow. The following commutation relations hold:
 $$
 [X,U] =U \,, \quad [X,V] = -V     \,, \quad   \quad [U,V] = 2X\,.
 $$

 \begin{theorem}
\label{thm:CE_horo} \cite{FF03}  There exists a  space ${\mathcal I}_U(M) $ of $U$-invariant distributions  of countable dimension such that for any $f \in W^s(M)$   (with $s>1$) such that
$$
D(f) =0 \, , \quad \text{ for all } \,D\in  {\mathcal I}_U(M)\,,
$$
is an $U$-coboundary with zero-average transfer function $u \in L^2(M, \omega)$. In addition, for all $t<s-1$
there exists a constant $C_{s,t}>0$ such that
$$
\vert u\vert_t \leq   C_{s,t}   \vert  f \vert_s \,.
$$
The space ${\mathcal I}_U(M)$ has a basis of generalized eigenvectors for the linear operator Lie derivative 
$\mathcal L_X$~along the geodesic flow with spectrum
$$
\sigma_{\mathcal I_U(M)} =\{  -\frac{1\pm \sqrt { 1- 4\mu}} {2}  \vert  \mu \in \sigma (\triangle)\} \, \bigcup \,( -\N) \,.
$$
The  linear  operator $\mathcal L_X$ on $\mathcal I_U(M)$ is diagonalizable, with eigenvalues of finite multiplicity, with the possible exception of finitely many $2\times 2$ Jordan blocks for the eigenvalue $1/4$ whenever $1/4 \in \sigma(\triangle)$.  The multiplicity of the eigenvalues is determined by the spectral
mutiplicities of the eigenvalues of the Laplace operator and by the dimensions of the spaces of holomorphic
$n$-differentials on the hyperbolic surface $S$.
\end{theorem} 

By the above result $k$-iterated coboundaries  coincide with the kernel of the space of $k$-invariant distributions
$$
\mathcal I_{U,k}(M) = \{  D \in \mathcal E'(M) \vert   U^k D=0\}\,.
$$

\begin{lemma} The space $\mathcal I_{U,k}(M)$ can be described as follows:
$$
\mathcal I_{U,k}(M)= \bigoplus_{j=0}^k   \mathcal L_V^j \mathcal I_U(M) \,.
$$
The space ${\mathcal I}_{U,k}(M)$ has a basis of generalized eigenvectors for the linear operator Lie derivative 
$\mathcal L_X$ along the geodesic flow with spectrum
$$
\sigma_{\mathcal I_{U,k}(M)} = \{0\} \cup   \bigcup_{j=0}^k \{ \lambda -j     \vert   
\lambda \in \sigma_{\mathcal I_{U}(M)}\setminus \{0\}      \}   \,.
$$
For every $\lambda \in  \sigma_{\mathcal I_{U}(M)}$ and every $j\in \{ 0, \dots, k \}$, let $E_X(\lambda -j)$ denote the generalized eigenspace of the operator $\mathcal L_X$ with eigenvalue $\lambda-j \in \C$. The operators 
$\mathcal L_U: E_X(\lambda-j-1) \to E_X(\lambda-j)$ and $\mathcal L_V: E_X(\lambda-j) \to  E_X(\lambda-j-1)$
are isomorphisms of finite dimensional vector spaces. The operator $\mathcal L_X$ is diagonalizable with the
exception of the eigenvalues  $1/4- j$, whenever $1/4 \in \sigma (\triangle)$, for which it has finitely many
$2\times 2$ Jordan blocks. 

\end{lemma}  
\begin{proof}

By the commutation relations we have
$$
[ \mathcal L_X,  \mathcal L_V] = - \mathcal L_V \quad \text{  and } \quad  [ \mathcal L_X,  \mathcal L_U] = \mathcal L_U \,.
$$
It follows that 
$$
\begin{aligned}
(\mathcal L_X -\lambda Id) \mathcal L_V &= [\mathcal L_X -\lambda Id,
  \mathcal L_V ]  \\ &+  \mathcal L_V   (\mathcal L_X -\lambda Id) =  \mathcal L_V   \left(\mathcal L_X -(\lambda+1) Id\right) \,,
\end{aligned}
$$
and by induction, for all $j\in \N$, 
\begin{equation}
\label{eq:annihilation}
(\mathcal L_X -\lambda Id)^j  \mathcal L_V =  \mathcal L_V   (\mathcal L_X -(\lambda+1) Id)^j \,.
\end{equation}
In fact, by induction hypothesis we have
$$
\begin{aligned}
(\mathcal L_X -\lambda Id)^j  \mathcal L_V &= (\mathcal L_X -\lambda Id)  
(\mathcal L_X -\lambda Id)^{j-1} \mathcal L_V  \\ &=   (\mathcal L_X -\lambda Id)  
\mathcal L_V   (\mathcal L_X -(\lambda+1) Id)^j \\ & = \mathcal L_V   (\mathcal L_X -(\lambda+1) Id)
(\mathcal L_X -(\lambda+1) Id)^{j-1} \\
&= \mathcal L_V (\mathcal L_X -(\lambda+1) Id)^j\,.
\end{aligned}
$$
Similarly, from the commutation relations we derive that, for all $j\in \N$, 
\begin{equation}
\label{eq:creation}
(\mathcal L_X -\lambda Id)^j  \mathcal L_U =  \mathcal L_V   (\mathcal L_X -(\lambda-1) Id)^j \,.
\end{equation}
By formulas \eqref{eq:annihilation} and \eqref{eq:creation} it follows that if $D \in \mathcal E'(M)$ 
is any distributional generalized eigenvector for the geodesic flow of eigenvalue $\lambda \in \C$
and algebraic multiplicity $m\geq 1$, then $\mathcal L_V D$ and $\mathcal L_U D$ are distributional 
generalized eigenvectors for the geodesic flow of eigenvalues respectively $\lambda-1 \in \C$ and 
$\lambda+1 \in \C$ and algebraic multiplicity $m\geq 1$. 

The Lie derivative horocyclic operators $\mathcal L_V$ and $\mathcal L_U$ act therefore as annihilation and creation operators on the distributional pure point spectrum of the geodesic flow Lie derivative operator $\mathcal L_X$, and they preserve the algebraic multiplicity of eigenvalues.

It follows immediately by the definition of the iterated invariant distributions that, for all $k\in \N\setminus\{0\}$, the operator $\mathcal L_U: \mathcal I_{U,k+1}(M) \to \mathcal I_{U,k}(M)$ is well-defined. We prove below that the operator 
$\mathcal L_V: \mathcal I_{U,k}(M) \to \mathcal I_{U,k+1}(M)$ is also well-defined.

By the commutation relations, it follows by induction that the operator $\mathcal L_X:  \mathcal I_{U,k}(M) \to \mathcal I_{U,k}(M)$ is well-defined. In fact, for any  $D \in \mathcal I_{U,k}(M)$ we have
$$
\begin{aligned}
\mathcal L_U^k \mathcal L_X D&= \mathcal L_U^{k-1} \mathcal L_U \mathcal L_X D =
\mathcal L_U^{k-1} [ \mathcal L_U, \mathcal L_X] D +   \mathcal L_U^{k-1}  \mathcal L_X \mathcal L_U  D
=  - \mathcal L_U^{k} D +  \mathcal L_U^{k-1}  \mathcal L_X \mathcal L_U  D\,.
\end{aligned}
$$
Thus for $k=1$ we immediately derive that, if $D \in \mathcal I_{U,1}(M)= \mathcal I_{U}(M)$, then $\mathcal L_U \mathcal L_X D =0$, hence $ \mathcal L_X D\in \mathcal I_{U,1}(M)$. In general, we have that, if $D \in \mathcal I_{U,k}(M)$, then  $\mathcal L_U  D\in \mathcal I_{U,k-1}(M)$
and, by the induction hypothesis, also $ \mathcal L_X \mathcal L_U  D \in \mathcal I_{U,k-1}(M)$. It then follows
by the above identity that  $\mathcal L_U^k \mathcal L_X D=0$, that is, $ \mathcal L_X D\in \mathcal I_{U,k}(M)$.
We have thus proved that  $\mathcal L_X:  \mathcal I_{U,k}(M) \to \mathcal I_{U,k}(M)$ is well-defined.

Next, we prove by induction that the operator $\mathcal L_V :  \mathcal I_{U,k}(M) \to \mathcal I_{U,k+1}(M)$ is well-defined. By the commutation relations we have
$$
\mathcal L^{k+1}_U  \mathcal L_V D = \mathcal L^{k}_U   \mathcal L_U\mathcal L_V D =
\mathcal L^{k}_U [\mathcal L_U,\mathcal L_V] D + \mathcal L^{k}_U  \mathcal L_V \mathcal L_U D
=2 \mathcal L^{k}_U \mathcal L_X D + \mathcal L^{k}_U  \mathcal L_V \mathcal L_U D   \,.
$$
Thus for $k=1$, if $D \in \mathcal I_{U,1}(M)= \mathcal I_{U}(M)$, since $\mathcal L_X D \in \mathcal I_{U,1}(M)$ we have $\mathcal L^{2}_U  \mathcal L_V D=0$, that is,   $\mathcal L_V D \in  \mathcal I_{U,1}(M)$. In general,
if $D \in  \mathcal I_{U,k}(M)$, then  $\mathcal L_X D \in  \mathcal I_{U,k}(M)$ and, by induction hypothesis,
$ \mathcal L_V \mathcal L_U D \in  \mathcal I_{U,k}(M)$. It follows by the above identity that $\mathcal L^{k+1}_U  \mathcal L_V D=0$, that is, $\mathcal L_V D \in  \mathcal I_{U,k+1}(M)$.  We have thus proved that  $\mathcal L_V:  \mathcal I_{U,k}(M) \to \mathcal I_{U,k+1}(M)$ is well-defined.

Finally, we prove that for any generalized eigenspace $E_X(\lambda)$ of the operator $\mathcal L_X$, transverse to the subspace $\mathcal I_U(M)$ of invariant distributions, the operators $\mathcal L_U:
E_X(\lambda) \to E_X(\lambda+1)$ and $\mathcal L_V :E_X(\lambda+1) \to E_X(\lambda)$ are isomorphisms
of finite dimensional vector spaces. By construction the operator $\mathcal L_U: \mathcal I_{U,k+1}(M) \to \mathcal I_{U,k}(M)$ is surjective and it is creation operator (which adds $+1$ to the spectrum). It follows that the restriction of $\mathcal L_U$ to every generalized eigenspace $E_X(\lambda)$ of $\mathcal L_X$ transversal to the subspace $\mathcal I_U(M)$ of invariant distribution is an isomorphism with range equal to the eigenspace $E(\lambda+1)$. The operator $\mathcal L_V$ is injective on each generalized eigenspace  $E_X(\lambda+1)$ of $\mathcal L_X$ orthogonal to constant functions, since the kernel of $\mathcal L_V$ on $E_X(\lambda+1)$ would consists of smooth invariant functions,  which have to be constant by the ergodicity of horocycle flows.  Hence, it is also an isomorphism. 

 \end{proof}

From the results on cohomological equations (Theorem \ref{thm:CE_horo}) and from the description of the space of iterated coboundaries and of the spectrum of the geodesic flow on them, we can derive the following statement of the
Ruelle resonances and Ruelle asymptotic for geodesic flows of compact hyperbolic surfaces.
  
  \begin{theorem} The set of Ruelle resonances of the geodesic flow $\phi^X_\R$ of a compact hyperbolic surface  is
$$
\{ 1Ę\}  \cup  \{ \exp \left(  -(\frac{1\pm \sqrt { 1- 4\mu}} {2}  -j) t\right)  \vert  \mu \in \sigma (\triangle)  \text{ and } j\in \N    \} 
\bigcup \{  e^{-n t} \vert  n\in \N\} \,.
$$
There are no Jordan blocks except for the eigenvalues with $\mu=1/4$, whenever 
 $1/4 \in \sigma (\triangle)$. In this case, there are finitely many $2\times 2$ Jordan blocks with eigenvalue
 $e^{- (\frac{1}{2} -j) t}$, for all $j\in \N$.  The Ruelle asymptotic takes the following form:
 $$
 \begin{aligned}
 \langle f \circ \phi^X_t, g\rangle &\approx \left(\int_M f d\text{vol} \right) \left(\int_M g d\text{vol} \right) +
 \sum_{\mu\in \sigma(\triangle)\setminus\{1/4\}} \sum_{\pm} \sum_{j\in \N} \mathcal C^{\pm}_{\mu,j} (f,g,t) e^{ -(\frac{1\pm \sqrt { 1- 4\mu}} {2}  -j) t } \\ & + \sum_{\mu \in \sigma(\triangle)\cap\{1/4\}} \sum_{\pm} \sum_{j\in \N} \mathcal C^+_{\mu,j} (f,g,t) e^{(-\frac{1}{2}-j)t } + C^+_{\mu,j} (f,g,t) t e^{(-\frac{1}{2}-j)t }
 + \sum_{n\in \N} \mathcal C_n (f,g,t) e^{-nt} \,.
\end{aligned}
 $$

\end{theorem}

A generalization of the above theorem to the geodesic flows on compact hyperbolic spaces in any dimensions has been carried out by S.~Dyatlov, F.~Faure and C.~Guillarmou in \cite{DFG}. 

\begin{problem} Extend the above theorem to geodesic flows on surfaces of non-constant negative curvature
and to Anosov flows in dimension $3$.
\end{problem}
 Partial or conditional results on this problem have been obtained by A. Adam \cite{Ad} and  F.~Faure and
 C.~Guillarmou~\cite{FG18}).

 \section{Ruelle resonances for (partially hyperbolic) Heisenberg automorphisms}
 \label{sec:heis}
 
 Let $N$ denote the $3$-dimensional Heisenberg group and let $M:= \Gamma \backslash N$ denote a Heisenberg  nilmanifold, that is, the quotient of  $N$ over a (necessarily co-compact) lattice $\Gamma< N$.
 
 Let $\Phi:M\to M$ denote a {\it partially hyperbolic}  automorphism of a Heisenberg nilmanifold $M$ and let 
 $\{X,Y,Z\}$ denote the
 corresponding Heisenberg frame of vector fields on $M$. The Heisenberg commutation relations hold:
 $$
 [X,Y] =Z \,, \quad [X,Z] = [Y,Z] =0\,.
 $$
 By the assumption, there exists $\lambda >1$ such that
$$
\Phi_*(X) = \lambda X,  \quad \Phi_*( Y) = \lambda^{-1} Y \quad \text{ and } \quad \Phi_*(Z) = Z \,.
$$
Let $\omega$ denote the $\Phi$-invariant volume-form.  The form $\omega$ is also invariant for all  nilflows on $M$.  For any pair $f, g$  of sufficiently smooth 
complex-valued functions on $M$, we are interested  in the asymptotic for the decay of the correlations 
$$
{\mathcal C}(f,g, n) =\langle f\circ \Phi^n, g \rangle_{L^2(M, \omega)}\,.
$$
As in the case of pseudo-Anosov diffeomorphisms, the key step is to characterize iterated coboundaries.  Coboundaries were
characterized in \cite{FF06} and \cite{FF07} (see also \cite{F14}). 

\begin{theorem} \cite{FF06}  There exists a  space ${\mathcal I}_X(M) $ of $X$-invariant distributions  of countable dimension
such that for any $f \in W^s(M)$   (with $s>1$) such that
$$
D(f) =0 \, , \quad \text{ for all } \,D\in  {\mathcal I}_X(M)\,,
$$
is an $X$-coboundary with zero-average transfer function $u \in L^2(M, \omega)$. In addition, for all $t<s-1$
there exists a constant $C_{s,t}>0$ such that
$$
\vert u\vert_t \leq   C_{s,t}   \vert  f \vert_s \,.
$$
\end{theorem}

\begin{lemma} 
\label{lemma:Heis_scaling}
\cite{FF06}  There exists a basis $\{D_{z,i}  \vert z \in \Z\setminus\{0\}, i =1, \dots, \vert z \vert \}$ of $\mathcal I_X(M)$ and
unit complex numbers $\{ u_{z,i} \vert  z \in \Z\setminus\{0\}, i =1, \dots, \vert z \vert \}$ such that, for all 
$z\in \Z\setminus\{0\}$  and  $i \in \{1, \dots, \vert z \vert\}$,
we have $Z D_{z,i} = \imath z D_{z,i}$,  and 
$$
\Phi_* (D_{z,i}) = u_{z,i} \lambda^{-1/2} D_{z,i} \,.
$$
\end{lemma}

The following result on the Ruelle resonances of partially hyperbolic Heisenberg automorphisms appears a special
case of in the work of F.~Faure and M.~Tsuji~\cite{FT15} (see Remark 1.3.5, (4)). Results on cohomological equation for the unstable flow from the methods of Faure and Tsujii, in the spirit
of the Giulietti, Liverani \cite{GL} and Faure, Gou\"ezel, Lanneau \cite{FGL},  were recently derived, in the linear case, by O.~Butterley and L.~Simonelli~\cite{BS}.  It is an interesting problem to generalize their work to the
case of partially hyperbolic non-linear of maps of Heisenberg nilmanifolds treated in the work of Faure and Tsujii \cite{FT15}.

\begin{theorem} 
\label{thm:RRes_Heis}
The set of Ruelle resonances of any partially hyperbolic Heisenberg automorphism is
$$
\{ 1Ę\}  \cup  \{ u_{z,i}  \lambda^{-k-1/2 } \vert z\in \Z\setminus\{0\}, i\in\{1, \dots, \vert z \vert\},  k \in \N\}\,.
$$
There exists a basis of the space of ``Ruelle eigenstates'' $\{D^{(k)}_{z,i}\}$ such that 
$$
\Phi_* ( D^{(k)}_{z,i}) = u_{z,i} \lambda ^{-k -1/2}  D^{(k)}_{z,i} + \sum_{j<k} c_{k,j}(\lambda)  u_{z,i} \lambda^{-j-1/2}  D^{(j)}_{z,i} \,.
$$ 
The coefficients $c_{k,j}(\lambda)$ are not all equal to zero.
\end{theorem}

\begin{proof}

We consider solutions $D^{(k)}_{z,i}$ of the iterated coboundary equations
$$
X^k D^{(k)}_{z,i} = D_{z,i} \,.
$$
By an immediate computation we have
$$
X^k [ \Phi_* (D^{(k)}_{z,i} ) - \lambda^{-k -1/2}   D^{(k)}_{z,i} ] =0\,.
$$

These solutions and the action of the automorphism on them can be explicitly computed by representation theory.
In fact, every irreducible representation space $H$ is unitarily equivalent  to $L^2(\R, dx)$ and the derived representation
$D\pi_z$ is given by the following formulas:
$$
D\pi_z (X) = \frac{d}{ dx}  \,, \quad  D\pi_z (Y) =  \imath z x   \quad \text{ and } \quad D\pi_z (Z) = \imath z I\,.
$$
Let $\mathcal F_H : H \to L^2(\R, dx)$ denote the unitary equivalence. 
The space of invariant distributions in each representation is $1$-dimensional, and it is generated by the functional
$$
D_H (f) = \int_\R  {\mathcal F}_H (f)(x) dx \,, \quad \text{ for all }\, f \in W^s_{X,Y,Z} (H) \subset H,  \text{ with }s>1/2.
$$
Since the Green operator $G_X$ of the cohomological equation $Xu=f$ can be written in each irreducible representation as
$$
[{\mathcal F}_HG_X (f)] (x) = \int_{-\infty} ^x {\mathcal F}_H(f) (\xi) d\xi =  - \int_{x} ^{+\infty} {\mathcal F}_H(f)(\xi) d\xi  \,,
$$
 the distributions $D^{(k)}_H$ are given {\it formally} by the formulas
 $$
D^{(k)}_H(f):=  \int_\R  \int_{-\infty}^{\xi_k  }\dots \int_{-\infty}^{\xi_1 }    \mathcal F_H(f) (\xi_0)      d\xi_0 d\xi_1\dots d\xi_k
 $$
In fact, the integral in the above formula for $D^{(k)}_H$ is convergent  (even for infinitely differentiable functions) only on the joint kernel of  $D^{(0)}_H, \dots, D^{(k-1)}_H$.

Let then $\{ \chi_a \}  \subset C^\infty_0(\R)$ be a system of  functions such that, for all $a \leq b $, 
$$
 \int_\R  \int_{-\infty}^{\xi_a }\dots \int_{-\infty}^{\xi_1 }    \chi_b (\xi_0)      d\xi_0 d\xi_1\dots d\xi_a = \delta_{ab} 
$$
and let $P^{(j)}_H:H \to L^2(\R, dx)$ denote the projectors recursively defined as 
$$
\begin{cases}
P^{(0)}_H  (f)  &= \mathcal F_H(f)   -   D_H(f)  \chi_0 \,, \\
P^{(j+1)}_H(f)  &= P^{(j)}_H(f)   -   D^{(j+1)}_H (P^{(j)}_H(f) ) \chi_{j+1} \,.
\end{cases}
$$
By construction we have that, for all $j\in \N$ and  for all $f \in W^s(H)$, $s >j+1/2$, 
$$
D^{(0)}_H(P^{(j)}_H (f)) = \dots = D^{(j)}_H(P^{(j)}_H (f)) =0\,.
$$
It follows that $D^{(j+1)}_H$ can be defined, for all $j\in \N$,  by the formula
$$
D^{(j+1)}_H(f) :=    \int_\R  \int_{-\infty}^{\xi_{j+1}  }\dots \int_{-\infty}^{\xi_1 }    P^{(j)}_H(f) (\xi_0)      d\xi_0 d\xi_1\dots d\xi_{j+1} \,.
$$
The volume preserving, hence unitary, action of the automorphism $\Phi$ on $L^2(M)$ preserves the isotypical components of the regular representation, that is, the eigenspaces of the central circle action.  For every
$z\in \Z\setminus\{0\}$,  there exists a splitting of the isotypical components $H_z$ into irreducible components
$$
H_z := \bigoplus_{i=1}^{\vert z \vert}  H_{z,i}
$$
such that for each $i\in \{1, \dots, \vert z \vert\}$ the component $H_{z,i}$ is invariant under the action of $\Phi$, hence it exists a unit complex number $u_{z,i} \in U(1)$ such that the operator $\Phi^*$ can be written in representation 
as follows:
$$
{\mathcal F}_{H_{z,i}} ( \Phi^* (f) ) (x) =  u_{z,i} \lambda^{1/2} {\mathcal F}_{H_{z,i}}(f) (\lambda x) \,, \quad \text{ for all } x \in \R \,.
$$
For every $z\in \Z\setminus\{0\}$, every $i \in \{1, \dots, \vert z \vert\}$ and every $j, k\geq 0$, let us adopt the notation
$$
{\mathcal F}_{z,i} := {\mathcal F}_{H_{z,i}}\,, \quad P^{(j)}_{H_{z,i}} := P^{(j)}_{z,i} \quad \text{ and } \quad  D^{(k)}_{z,i} := D^{(k)}_{H_{z,i}}    \,.
$$

A direct calculation then shows that, as claimed,
$$
\Phi_* ( D^{(k)}_{z,i}) =u_{z,i} \lambda ^{-k -1/2}  D^{(k)}_{z,i} + \sum_{j<k} c_{k,j}(\lambda)  u_{z,i}  \lambda^{-j-1/2}  D^{(j)}_{z,i} \,.
$$ 
The constants $c_{k,j}(\lambda)$ can be explicitly computed in representation. 

For instance, for all $f\in W^\infty(H_{z,i})$ we have
$$
D^{(0)}_{z,i}(f\circ \Phi) =  \int_\R  u_{z,i} \lambda^{1/2} \mathcal F_{z,i}(f) (\lambda x) dx = u_{z,i} \lambda^{-1/2} D^{(0)}_{z,i}(f)\,,
$$
hence the invariant distribution $D^{(0)}_{z,i}$ is indeed an eigendistribution.

For $j=1$ by definition we have 
$$
\begin{aligned}
D^{(1)}_{z,i} ( f\circ \Phi) &= \int_\R \int_{-\infty}^x  [u_{z,i} \lambda^{1/2} \mathcal F_{z,i}(f) (\lambda \xi) - u_{z,i} \lambda^{-1/2} D^{(0)}_{z,i} (f) \chi_0(\xi)]  d\xi dx \\
&= u_{z,i} \int_\R \int_{-\infty}^x  \lambda^{1/2} [ \mathcal F_{z,i}(f) (\lambda \xi) - D^{(0)}_{z,i} (f) \chi_0(\lambda\xi)]  d\xi dx  \\
&+ u_{z,i} \lambda^{-1/2}   D^{(0)}_{z,i} (f)   \int_\R \int_{-\infty}^x [\lambda \chi_0(\lambda\xi)- \chi_0(\xi)] d\xi dx
\end{aligned}
$$
By change of variable we have
$$
\begin{aligned}
\int_\R \int_{-\infty}^x  \lambda^{1/2} [ \mathcal F_{z,i}(f) (\lambda \xi) - D^{(0)}_{z,i} (f) \chi_0(\lambda\xi)]  d\xi dx \\
= \lambda^{-3/2}  \int_\R \int_{-\infty}^x  [ \mathcal F_{z,i}(f) ( \xi) - D^{(0)}_{z,i} (f) \chi_0(\xi)]  d\xi dx =  \lambda^{-3/2} D^{(1)}_{z,i} (f)
\end{aligned}
$$
 and 
$$
\begin{aligned}
c_{1,0}(\lambda) :=   \int_\R \int_{-\infty}^x [\lambda \chi_0(\lambda\xi)- \chi_0(\xi)] d\xi dx =  \int_\R \int_{x}^{\lambda x} \chi_0(\xi) d\xi dx \,.
\end{aligned}
$$
We therefore have that, as claimed,
$$
\Phi_* (D^{(1)}_{z,i}) = u_{z,i} \lambda^{-3/2} D^{(1)}_{z,i} + c_{1,0}(\lambda) u_{z,i}  \lambda^{-1/2}   D^{(0)}_{z,i}\,.
$$
However, we remark that when $\chi_0$ is chosen to be an even function, then
$$
c_{1,0}(\lambda)= \int_\R \int_{x}^{\lambda x} \chi_0(\xi) d\xi dx =0\,.
$$
By definition we then have $P^{(1)}_{z,i}(f) =P^{(0)}_{z,i}(f) - D^{(1)}_{z,i} (f) \chi_1$ with
$$
\int_\R \chi_1(\xi_0) d\xi_0 =0  \,, \quad \text{ and } \quad    \int_\R \int_{-\infty}^{\xi_1}  \chi_1(\xi_0) d\xi_0  d\xi_1 =1\,.
$$
It follows that 
$$
P^{(1)}_{z,i}(f\circ \Phi)(x) = u_{z,i} \lambda^{1/2} \mathcal F_{z,i}(f)(\lambda x)  - u_{z,i} \lambda^{-1/2}  D^{(0)}_{z,i} (f) \chi_0 -   u_{z,i} \lambda^{-3/2}  D^{(1)}_{z,i} (f) \chi_1\,.
$$ 
 Up to the unit complex factor $u_{z,i}$, we can write $P^{(1)}_{z,i}(f\circ \Phi)(x)$ as the sum of three terms. The first term is 
 $$
 \lambda^{1/2}[  \mathcal F_{z,i}(f) - D^{(0)}_{z,i} (f) \chi_0  - D^{(1)}_{z,i} (f) \chi_1] (\lambda x)\,,
 $$
 which after triple integration contributes to $D^{(2)}_{z,i} (f\circ \Phi)$ a term
 $$
 \lambda ^{1/2- 3} D^{(2)}_{z,i} (f) \\ =
 \int_\R\int_{-\infty}^{\xi_2} \int_{-\infty}^{\xi_1} \lambda^{1/2}[  \mathcal F_{z,i}(f) - \sum_{a=0,1} D^{(a)}_{z,i} (f) \chi_a ] (\lambda \xi_0) d\xi_0 d\xi_1 d\xi_2 
  \,.
 $$
 The second terms is equal to $\lambda^{-1/2}  D^{(0)}_{z,i} (f) ( \lambda \chi_0(\lambda x) - \chi_0(x) )$.    The corresponding coefficient
 $$
 c_{2,0}(\lambda) = \int_\R\int_{-\infty}^{\xi_2} \int_{-\infty}^{\xi_1}  [\lambda \chi_0(\lambda \xi_0) - \chi_0(\xi_0)] d\xi_0 d\xi_1 d\xi_2  \not = 0\,.
 $$
 The third term is equal to $\lambda^{-3/2}  D^{(1)}_{z,i} (f) ( \lambda^2 \chi_1(\lambda x) - \chi_1(x) )$. The corresponding coefficient
 $$
 c_{2,1}(\lambda) = \int_\R\int_{-\infty}^{\xi_2} \int_{-\infty}^{\xi_1}  [\lambda^2 \chi_1(\lambda \xi_0) - \chi_1(\xi_0)] d\xi_0 d\xi_1 d\xi_2  = 0\,.
 $$
It follows that we have
$$
\Phi_*(D^{(2)}_{z,i}) = u_{z,i} \lambda ^{-1/2- 2} D^{(2)}_{z,i} +   c_{2,0}(\lambda)  u_{z,i} 
\lambda^{-1/2}  D^{(0)}_{z,i}\,.
$$
 
\end{proof}

\section {Transfer cocycles and generic nilflows} 
\label{sec:transferheis}

In this section we describe a transfer cocycle adapted to generic nilflows on Heisenberg nilmanifolds.
We refer to the paper~\cite{FF06} and to the survey~\cite{F14} for additional details. 

The deformation of Heisenberg structures on a Heisenberg nilmanifold $M= \Gamma \backslash N$ is
the space $\mathcal D_M$  of all Heisenberg frames $\{X,Y, Z\}$, that is, all frames such that
$$
[X,Y]=Z \,,  \quad [X,Z]=[Y,Z]=0\,.
$$
It can be proved that there exists a isomomorphism between $\mathcal D_M$ and the the group
$\text{Aut}_Z(N)$ of automorphisms of $N$ which fix the central vector field $Z$. The automorphism
is not canonical, since it depends on the choice of a base point $\{X_0, Y_0, Z\}$. Given any $a\in 
\text{Aut}_Z(N)$ the frame $\{a_*(x_0), a_*(Y_0), a_*(Z)\}$ is a Heisenberg frame, and the map
$$
a \to \{a_*(X_0), a_*(Y_0), a_*(Z)\}= \{a_*(x_0), a_*(Y_0), Z\}
$$
defines an isomorphism of $\text{Aut}_Z(N)$ onto $\mathcal D_M$. The subgroup of coordinate changes is the subgroup $\text{Aut}_\Gamma(N) < \text{Aut}_Z(N)$ of automorphisms which also fix the lattice $\Gamma$.
Each element of $\text{Aut}_\Gamma(N)$ induces a diffemorphism on $M$.  

The group $\text{Aut}_Z(N)$
acts on itself by right or left multiplication. 
The moduli space of Heisenberg frames on $M$ is the space $\mathcal M :=  \text{Aut}_\Gamma(N) 
\backslash \text{Aut}_Z(N)$. We note that $\text{Aut}_Z(N)$ is isomorphic to a $SL(2, \R) \times \R^2$
and $\text{Aut}_\Gamma(N)$ to a finite index subgroup of $SL(2, \Z) \ltimes \Z^2$, so that $\mathcal M$ is 
isomorphic to a toral bundle (with fiber $\T^2$) over a finite cover of $SL(2, \Z) \backslash SL(2,\R)$, the
unit tangent bundle of the modular surface.

The renormalization flow of Heisenberg nilflows is defined as follows: let $a_\R$ denote the one-parameter
group defined as follows:
$$
a_t(X_0, Y_0, Z)= (e^t X_0, e^{-t} Y_0, Z), \quad \text{ for all } t\in \R\,.
$$
The group $a_\R$ acts on $ \text{Aut}_Z(N)$, hence on the moduli space $\mathcal M$, by right multiplication.
In terms, of Heisenberg triple the action of $a_\R$ can be described as follows:
$$
a_t(X,Y, Z) =  (e^tX, e^{-t} Y, Z) , \quad \text{ for all } \{X,Y,Z\} \in \mathcal D \text{ and } t\in \R.
$$
In fact, if $(X,Y,Z)= a_\ast (X_0, Y_0, Z)$ then 
$$
(a a_t)_\ast(X_0, Y_0, Z)= a_\ast (e^t X_0, e^{-t} Y_0, Z) = (e^tX, e^{-t} Y, Z)\,.
$$
We note that by definition the renormalization flow projects onto the hyperbolic geodesic flow (diagonal flow) on 
a finite cover of $SL(2, \Z) \backslash SL(2,\R)$.  It can be proved (see \cite{FF06}) that $a_\R$ is an Anosov
flow of the $5$-dimensional moduli space. In addition, there is a one-to-one correspondence between 
periodic orbits of $a_\R$ and partially hyperbolic Heisenberg automorphisms, similar to the one between periodic
orbits of the Teichm\"uller flow. In fact, if  $(X,Y, Z)= a_\ast (X_0, Y_0, Z)$ is a periodic point for $a_\R$, there
exists a $T>0$ and an element $\Phi\in \text{Aut}_\Gamma(N)$  such that
$ a a_T = \Phi a$, hence
$$
\Phi (X,Y, Z) =  (e^T X, e^{-T} Y, Z)\,, 
$$
hence $\Phi$ induces a partially hyperbolic automorphism of $M$. 

We define, for every $s\in \R$, a Sobolev bundle $W^s(M)$ over the moduli space $\mathcal M$. The
Sobolev bundle $W^s(M)$ is the projection to $\mathcal M$ of the bundle over the deformation space
$\mathcal D_M$ with fiber the Sobolev space $W^s_{X,Y,Z} (M)$ at every point $\{X,Y, Z\} \in \mathcal D_M$.
The Sobolev space $W^s_{X,Y,Z} (M)$ is defined as the domain of the self-adjoint operator $(I- (X^2+ Y^2+Z^2))^{s/2}$ endowed with the graph norm. 
For $s>0$, we also define the sub-bundle $\mathcal I^s (M) \subset W^{-s}(M)$ of $X$-invariant distributions,
with fiber $\mathcal I^s_{X,Y,Z} (M) \subset W^{-s}_{X,Y,Z}(M)$ the subspace 
$$
\mathcal I^s_{X,Y,Z} (M) := \{ D\in W^{-s}_{X,Y,Z}(M) \vert  \mathcal L_XD=0 \} \,.
$$
 We define the renormalization cocycle $\rho_\R$ on the bundle 
$W^s(M)$ over $\mathcal M$ as the projection of the trivial cocycle on $W^s(M)$ over $\mathcal D_M$, 
that is, of the cocycle 
$$\text{Id} :  W^s_{X,Y,Z} (M) \to W^s_{a_t(X,Y,Z) } (M)\,, $$
given by the identification of the vector spaces $W^s_{X,Y,Z} (M)$ and $W^s_{a_t(X,Y,Z) } (M)$. 

Our main result in \cite{FF06} is a statement on the Lyapunov spectrum of the restriction of the renormalization 
cocycle  $\rho_\R$ to the bundle $\mathcal I^s (M)$.

\begin{theorem} 
\label{thm:Heis_Lyap}
\cite{FF06} For any $s>1/2$,  the Lyapunov spectrum of the cocycle $\rho_\R \vert \mathcal I^s (M)$ (with respect to any probability $a_\R$-invariant measure  on the moduli space $\mathcal M$) consists, in addition to the simple exponent $1$,  of the single Lyapunov exponent $1/2$ with infinite multiplicity. 
\end{theorem} 
The above theorem generalizes Lemma~\ref{lemma:Heis_scaling} from periodic orbits to all probability 
invariant measures.  As a consequence of Theorem \ref{thm:Heis_Lyap}  we have the following result on
the deviation of ergodic averages of nilflows.

\begin{corollary}  For almost all $\{X,Y,Z)\}$ with respect to any $a_\R$-invariant measure the following holds.
For any $s>5/2$, and for any $\epsilon >0$, there exists a constant $C_{s, \epsilon}(X,Y,Z)$ such that, for
any function $f \in W^s(M)$ of zero average, for any $(x,T) \in M \times \R^+$, 
$$
\left\vert \int_0^T f \circ \phi^X_t (x) dt \right\vert \leq C_{s, \epsilon}(X,Y,Z) \Vert f \Vert_s T^{1/2+ \epsilon}\,.
$$
\end{corollary} 

In the spirit of these notes, we can derive from Theorem~\ref{thm:Heis_Lyap} a result on the Lyapunov spectrum
of $\rho_\R$ on the bundle $\mathcal I^s_k(M)$ of iterated invariant distribution, for all $k\in \N$, and on the bundle $\mathcal I^s_\infty(M)$. For instance,  let $\mathcal I_\infty(M)  $ denote 
the subbundle with fiber
$$
\mathcal I_{X,\infty}(M) = \bigcup_{k\in \N} \{D\in \mathcal E'(M) \vert  X^k D =0\}\,
$$
\begin{theorem} 
\label{thm:Heis_Lyap_iterated}
The Lyapunov spectrum of the cocycle $\rho_\R \vert \mathcal I_\infty (M)$ (with respect to any probability $a_\R$-invariant measure  on the moduli space $\mathcal M$) consists of the set
$$
\{1\} \cup \{ 1/2- k \vert k\in \N\}\,.
$$
The Lyapunov exponent $1$ is simple and corresponds to the subbundle of $\mathcal I_\infty (M)$ given by the invariant volume on $M$, while all the Lyapunov exponents $1/2-k$ have infinite multiplicity.
\end{theorem} 
The above theorem generalizes  Theorem~\ref{thm:RRes_Heis} from periodic orbits to all probability 
invariant measures and can be derived from Theorem~\ref{thm:Heis_Lyap} by a similar argument.

We conclude with the following
\begin{problem}  Prove a version of the above theorem for a transfer cocycle over a bundle of anisotropic currents
over the renormalization flow on the moduli space of Heisenberg manifolds.
\end{problem}



 \end{document}